\documentclass[11pt]{article}
\usepackage{amsmath} 
\usepackage{verbatim}
\usepackage{amssymb, float}
\usepackage{color}
\usepackage{graphicx,epsfig}
\graphicspath{{figures/}}
\usepackage[applemac]{inputenc}
\title{Lyapunov stability
of differential inclusions with Lipschitz Cusco perturbations of maximal
monotone operators}
\newtheorem{theorem}{Theorem}[section]
\newtheorem{corollary}[theorem]{Corollary}
\newtheorem{lemma}[theorem]{Lemma}
\newtheorem{proposition}[theorem]{Proposition}
\newtheorem{definition}[theorem]{Definition}

\def\beq{\begin{equation}}
\def\eeq{\end{equation}}
\def\baq{\begin{eqnarray}}
\def\eaq{\end{eqnarray}}
\def\baqn{\begin{eqnarray*}}
\def\eaqn{\end{eqnarray*}}

\def\image #1 (#2,#3) (echelle #4) #5{
\dimen2=#2
\dimen3=#3
\divide \dimen2 by 1000
\multiply \dimen2 by #4
\divide \dimen3 by 1000
\multiply \dimen3 by #4
\setbox1 =\vbox to \dimen2{\hsize=\dimen3\vfill\special{picture #1
scaled #4}}
\vbox{\hsize=\dimen3\box1\medskip\centerline{#5}}
}

\begin{document}
\author{
Samir {\sc Adly}\thanks{Laboratoire XLIM, Universit\'e de Limoges, France. Email: samir.adly@unilim.fr.}, Abderrahim \textsc{Hantoute}\thanks{Center for Mathematical Modeling (CMM), Universidad de Chile. Email: ahantoute@dim.uchile.cl.}, Bao Tran \textsc{Nguyen}\thanks{Universidad de O'Higgins, Rancagua, Chile. Email: nguyenbaotran31@gmail.com }
}

\maketitle
\begin{abstract}\noindent We give criteria for weak and strong invariant closed sets for differential
inclusions given in $\mathbb{R}^{n}$ and governed by Lipschitz Cusco
perturbations of maximal monotone operators. Correspondingly, we provide
different characterizations for the\ associated strong Lyapunov functions.
The resulting conditions only depend on the data of the system.\end{abstract}

\setcounter{tocdepth}{1}
\tableofcontents


\section{Introduction\label{sec1}}

In this paper we investigate\ (weak and strong) invariant closed sets with
respect to the following differential inclusion, given in $\mathbb{R}^{n}$, 
\begin{equation}
\dot{x}(t)\in F(x(t))-A(x(t))\,\,\,\mbox{\ a.e.\ }t\geq 0,\text{ }%
x(0)=x_{0}\in \overline{\text{dom}A},  \label{ieq}
\end{equation}%
where $F:\mathbb{R}^{n}\rightrightarrows \mathbb{R}^{n}$ is a Lipschitz
Cusco multifunction; that is, a Lipschitz set-valued mapping with nonempty,
convex and compact values, and $A:\mathbb{R}^{n}\rightrightarrows \mathbb{R}%
^{n}$ is a maximal monotone operator. There is no restriction on the initial
condition $x_{0}$ that can be any point in the closure of the domain of $A,$
possibly not a point of definition of $A.$ Equivalently, we also
characterize (strong) Lyapunov functions and, more generally, $a$-Lyapunov
pairs associated to the differential inclusion above. Our criteria are given
by means only of the data of the system, represented by\ the multifunction $F
$ and the operator $A,$ together with first-order approximations of the
invariant sets candidates, using Bouligand tangent cones, or, equivalently,
Fr\'{e}chet or proximal normal cones, and first-order (general) derivatives
of Lyapunov functions candidates, using directional derivatives, Fr\'{e}chet
or proximal subdifferentials.\ 

Our analysis aims at gathering in one framework two different kinds of
dynamic systems that were studied separately in the literature, at least in
what concerns Lyapunov stability. The first kind of these dynamic systems is
governed exclusively by Cusco multifunctions, and gives rise to a natural
extension of the classical differential equations, given in the form 
\begin{equation}
\dot{x}(t)\in F(x(t))\,\,\,\mbox{\ a.e.\ }t\geq 0,\text{ }x(0)=x_{0}\in 
\mathbb{R}^{n}.  \label{ieq'.}
\end{equation}%
The consideration of differential inclusions rather than differential
equations allows more useful existence theorems, as revealed by\ Filippov's
theory for differential equations with discontinuous right-hand-sides \cite%
{F}. Stability of such systems; namely, the study of Lyapunov functions and
invariant sets, has been extensively studied and investigated especially
during the nineties by many authors; see, for example, \cite{CLSW, CLR, DRW}%
, as well as \cite{AF, A, FP} (see, also, the references therein). For
instance, complete characterizations for closed sets can be found in \cite%
{CLSW} in the finite-dimensional setting, and in \cite{CLR} for Hilbert
spaces. It is worth recalling that only the upper semicontinuity of the
Cusco mapping $F$ is required for the weak invariance, while Lipschitzianity
is used for the strong invariance (see \cite{CLR}). Invariance
characterizations of a same nature have been done in \cite{DRW} for\
one-side Lipschitz (not necessary Lipschitz) and compact-valued
multifunctions. These results have been adapted in \cite{CP2} to the
following more general differential inclusion (for $T\in \lbrack 0,+\infty ]$%
)%
\begin{equation}
\dot{x}(t)\in F(t,x(t))-\mathrm{N}_{C(t)}(x(t))\,\,\,\text{a.e.}\,\,t\in
\lbrack 0,T],\,x(0)=x_{0}\in C(0),  \label{ieq''}
\end{equation}%
where $C(t)$ is a uniformly prox-regular sets in $\mathbb{R}^{n}$ and $%
\mathrm{N}_{C(t)}$ is the associated normal cone. Observe here that the
right-hand-side may be unbounded, but, however, in the case when $T<+\infty $%
, the last differential inclusion above is equivalent to the following one,
for some positive constant $M>0,$ 
\begin{equation*}
\dot{x}(t)\in F(t,x(t))-\mathrm{N}_{C}(t)(x(t))\cap \mathrm{B}(\theta
,M)\,\,\,\text{a.e.}\,\,t\in \lbrack 0,T],x(0)=x_{0}\in C(0),
\end{equation*}%
giving rise to a differential inclusion in the form of (\ref{ieq'.}).

The other kind of systems that is covered by (\ref{ieq}) concerns
differential inclusions governed by maximal monotone operators, or, more
generally, (single-valued) Lipschitz perturbations of these operators, that
we write as 
\begin{equation}
\dot{x}(t)\in f(x(t))-A(x(t))\,\,\,\mbox{\ a.e.\ }t\geq 0,\text{ }%
x(0)=x_{0}\in \overline{\text{dom}A}.  \label{mono}
\end{equation}%
This system can be seen as perturbations of the ordinary differential
equation $\dot{x}(t)=f(x(t)),$ where $A$ could represent some associated
control action. In this single-valued Lipschitzian setting, weak and strong
invariance coincide since differential inclusion (\ref{ieq}) possesses
unique solutions. Compared to (\ref{ieq'.}) the right-hand-side in this
differential inclusion can be unbounded, or even empty. Typical examples of (%
\ref{mono}) involve the Fenchel subdifferential of proper, lower
semicontinuous convex functions (\cite{AG}). System (\ref{mono})\ has been
extensively studied; namely, regarding existence, regularity and properties
of the solutions \cite{Br}, while\ Lyapunov stability of such systems have
been initiated in \cite{P}; see, also, \cite{AHB2, AHT1, AHT2} for recent
contributions on the subject. Different criteria using the semi-group
generated by the operator $A$ can also be found in\ \cite{KS}, where
Lyapunov functions are characterized as viscosity-type solutions of
Hamilton-Jacobi equations, and\ in \cite{CM}, using implicit tangent cones
associated to the invariant sets candidates.

It is worth observing that (\ref{ieq}) is a special case of the following
more general differential inclusion 
\begin{equation}
\dot{x}(t)\in F(t,x(t))-A(t)(x(t))\,\,\,\mbox{\ a.e.\ }t\geq 0,\text{ }%
x(0)=x_{0}\in \overline{\text{dom}A(0,\cdot )},  \label{ieq'}
\end{equation}%
where $A$ and $F$ are also allowed to move in an appropriate way with
respect to the time variable, and satisfying some natural continuity and
measurability conditions. Existence of solution of (\ref{ieq'}) have been
also studied in \cite{AS,LM, SY} among others. In particular, in the Hilbert
spaces setting, \cite{AS} considered similar systems as the one in (\ref{ieq}%
), but with requiring\ strong assumptions on the multifunction $F.$ In \cite%
{SY} the authors assume that $F$ is a single-valued mapping, that is
Lipschitz with respect to the second variable, while the minimal section
mapping of the maximal monotone operators $A(t)$ is uniformly bounded.

In this paper, we study and characterize strong and weak invariant closed
subsets of the closure of the domain of $A,$ dom$A,$ with respect to
differential inclusion (\ref{ieq}). We shall assume in our analysis that the
invariant sets candidates $S\subset \mathbb{R}^{n}$ satisfy the following
condition 
\begin{equation}
\Pi _{S}(x)\subset S\cap \text{dom$A$}\,\,\,\forall x\in \text{dom$A$},
\label{cond.}
\end{equation}%
where $\Pi _{S}$ refers to the projection operator on $S.$ This condition
has been used in many works; see, for instance, \cite{BP}, where the author
is concerned with flow invariance characterizations for differential
equations, with right-hand-sides given by nonlinear semigroup generators in
the sense of Crandall- Liggett (see \cite{CL}). It is clear that condition (%
\ref{cond.}) holds whenever $S\subset \text{dom$A$}$. When dealing with weak
invariant closed sets, we shall require some usual\ boundedness conditions\
on the invariant set, relying on\ the minimal norm section of the maximal
monotone operator $A.$ 

The paper is organized as follows: After Section 2, reserved to give the
necessary notations and present the main tools, we make in Section 3 a
review of the existence theorems of differential inclusion (\ref{ieq}), and
establish some first properties of the solutions. In Section 4 we
characterize weak and strong invariant closed sets with respect to (\ref{ieq}%
), while in Section 5, criteria for strong Lyapunov pairs are provided. 

\section{Notation and main tools}

In this paper, the notations $\langle \cdot ,\cdot \rangle $ and $\left\Vert
\cdot \right\Vert $ are the inner product and the norm in $\mathbb{R}^{n}$,
respectively. For each $x\in \mathbb{R}^{n}$ and $\rho \geq 0,$ $\mathrm{B}%
(x,\rho )$ is the closed ball with center $x$ and radius $\rho $; in
particular, we denote $\mathrm{B}_{r}:=\mathrm{B}(\theta ,r)$ where $\theta $
is the origin vector\ in $\mathbb{R}^{n}$. Given a nonempty set $S\subset 
\mathbb{R}^{n}$, we denote by $\overline{S}$ and int$(S)$ the closure and
the interior of $S,$ respectively. We denote by $\left\Vert S\right\Vert $
the nonnegative real number define by 
\begin{equation*}
\left\Vert S\right\Vert :=\sup \{\left\Vert v\right\Vert :\,\,v\in S\}.
\end{equation*}%
The \textit{orthogonal projection mapping} onto $S$ is defined as 
\begin{equation*}
\Pi _{S}(x):=\{s\in S:\,\,\left\Vert x-s\right\Vert =d_{S}(x)\},
\end{equation*}%
where $d_{S}(x):=\inf \{\left\Vert x-s\right\Vert ,s\in S\}$ is the \textit{%
distance function} to $S$. If $S$ is a closed set, then $\Pi _{S}(x)\neq
\emptyset $ for every $x\in \mathbb{R}^{n}$. We denote by $S^{0}:=\Pi
_{S}(\theta )$ the minimal norm vector in $S$. The \textit{indicator function%
} of $S$ is defined as 
\begin{equation*}
\mathrm{I}_{S}(x):=%
\begin{cases}
0 & \text{if}\,\,x\in S \\ 
+\infty & \text{if}\,\,x\notin S,%
\end{cases}%
\end{equation*}%
and the \textit{support function} of $S$ is defined as 
\begin{equation*}
\sigma _{S}(x):=\sup \{\langle x,s\rangle :s\in S\},
\end{equation*}%
with the convention that $\sigma _{\emptyset }=-\infty .$ Given a function $%
\varphi :\mathbb{R}^{n}\rightarrow \mathbb{R}\cup \{+\infty \}$, its \textit{%
domain} and \textit{epigraph} are defined by 
\begin{equation*}
\begin{split}
& \text{dom}\varphi :=\{x\in \mathbb{R}^{n}:\varphi (x)<+\infty \}; \\
& \text{epi}\varphi :=\{(x,\alpha )\in \mathbb{R}^{n+1}:\varphi (x)\leq
\alpha \}.
\end{split}%
\end{equation*}%
We say $\varphi $ is proper if $\text{dom}\varphi \neq \emptyset $; lower
semicontinuous (lsc for short), if $\text{epi}\varphi $ is closed. We denote
by\ $\mathcal{F}(\mathbb{R}^{n})$ the set all proper and lsc functions.%
\newline
Next, we introduce some basic concepts of nonsmooth and variational
analysis. Let $\varphi \in \mathcal{F}(\mathbb{R}^{n})$ and $x\in \text{dom}%
\varphi .$ We call $\xi \in \mathbb{R}^{n}$ a \textit{proximal subgradient}
of $\varphi $ at $x$, written $\xi \in \partial _{P}\varphi (x),$ if 
\begin{equation*}
\underset{y\rightarrow x,y\neq x}{\lim \inf }\frac{\varphi (y)-\varphi
(x)-\langle \xi ,y-x\rangle }{\left\Vert y-x\right\Vert ^{2}}>-\infty .
\end{equation*}%
A vector $\xi \in \mathbb{R}^{n}$ is said to be a \textit{Fr\'echet 
subgradient} of $\varphi $ at $x$, written $\xi \in \partial _{P}\varphi (x)$
if 
\begin{equation*}
\underset{y\rightarrow x,y\neq x}{\lim \inf }\frac{\varphi (y)-\varphi
(x)-\langle \xi ,y-x\rangle }{\left\Vert y-x\right\Vert }\geq 0.
\end{equation*}%
The \textit{limiting subdifferential} of $\varphi $ at $x$ is defined as 
\begin{equation*}
\partial _{L}\varphi (x):=\{\underset{n\rightarrow \infty }{\lim }\xi
_{n}\,\,|\,\,\xi _{n}\in \partial _{P}\varphi (x_{n}),x_{n}\rightarrow
x,\varphi (x_{n})\rightarrow \varphi (x)\},
\end{equation*}%
and the singular subdifferential of $\varphi $ at $x$ as 
\begin{equation*}
\partial _{\infty }\varphi (x):=\{\underset{n\rightarrow \infty }{\lim }%
\alpha _{n}\xi _{n}\,\,|\,\,\xi _{n}\in \partial _{P}\varphi
(x_{n}),x_{n}\rightarrow x,f(x_{n})\rightarrow f(x),\alpha _{n}\downarrow
0\}.
\end{equation*}%
The \textit{Clarke subdifferential} of $\varphi $ at $x$ is 
\begin{equation*}
\partial _{C}\varphi (x):=\overline{\text{co}}\{\partial _{L}\varphi
(x)+\partial _{\infty }\varphi (x)\}.
\end{equation*}%
In the case $x\notin \text{dom}\varphi $, by convention we set $\partial
_{P}\varphi (x)=\partial _{F}\varphi (x)=\partial _{L}\varphi (x)=\emptyset
. $ We have the classical inclusions $\partial _{P}\varphi (x)\subset
\partial _{F}\varphi (x)\subset \partial _{L}\varphi (x).$ If $\varphi $ is
locally Lipschitz around $x$, then $\partial _{\infty }\varphi (x)=\{\theta
\}$ and 
\begin{equation*}
\partial _{C}\varphi (x)=\overline{\text{co}}\partial _{L}\varphi (x).
\end{equation*}%
The \textit{generalized directional derivative} of $\varphi $ at $x$ in the
direction $v$ is defined by 
\begin{equation*}
\varphi ^{0}(x;v):=\underset{y\rightarrow x,t\downarrow 0}{\limsup }\frac{%
\varphi (y+tv)-\varphi (y)}{t}.
\end{equation*}%
We have that 
\begin{equation*}
\varphi ^{0}(x;v)=\underset{\xi \in \partial _{C}\varphi (x)}{\sup }\langle
\xi ,v\rangle \,\,\forall v\in \mathbb{R}^{n}.
\end{equation*}%
We also remind the \textit{contingent directional derivative} of $\varphi $
at $x\in \text{dom}\varphi $ in the direction $v\in \mathbb{R}^{n}$, which
is given by 
\begin{equation*}
\varphi ^{\prime }(x;v):=\underset{t\rightarrow 0^{+},w\rightarrow v}{\lim
\inf }\frac{\varphi (x+tw)-\varphi (x)}{t}.
\end{equation*}%
From the definition of the proximal and the Fr\'echet subdifferentials, it
is easy to prove that 
\begin{equation}
\sigma _{\partial _{P}\varphi (x)}(\cdot )\leq \sigma _{\partial _{F}\varphi
(x)}(\cdot )\leq \varphi ^{\prime }(x;\cdot )\,\,\forall x\in \text{dom}%
\varphi .  \label{1.1}
\end{equation}%
The \textit{proximal, the Fr\'echet, and the limiting normal cones} are
defined, respectively, by 
\begin{equation*}
\mathrm{N}_{S}^{P}(x):=\partial _{P}\mathrm{I}_{S}(x),\text{ }\mathrm{N}%
_{S}^{F}(x):=\partial _{F}\mathrm{I}_{S}(x),\text{ \ }\mathrm{N}%
_{S}^{L}(x):=\partial _{L}\mathrm{I}_{S}(x).
\end{equation*}%
We also define the singular prox-subdifferential $\partial _{P,\infty
}\varphi (x)$ of $\varphi $ at $x$ as those elements $\xi \in \mathbb{R}^{n}$
such that 
\begin{equation*}
(\xi ,0)\in \mathrm{N}_{\text{epi}\varphi }^{P}(x,\varphi (x)).
\end{equation*}%
The \textit{Bouligand tangent cones to} $S$ at $x$ is defined as 
\begin{equation*}
\mathrm{T}_{S}^{B}(x):=\left\{ v\in H\,|\,\exists \,\,x_{k}\in S,\exists
\,\,t_{k}\rightarrow 0,\,\,\text{st}\,\,t_{k}^{-1}(x_{k}-x)\rightarrow v%
\mbox{ as }k\rightarrow +\infty \right\} .
\end{equation*}

Next we recall some basic concepts and properties of maximal monotone
operators. For a multivalued operator $A:\mathbb{R}^{n}\rightrightarrows 
\mathbb{R}^{n}$, the \textit{domain} and the \textit{graph} are given,
respectively, by 
\begin{equation*}
\text{dom}A:=\{x\in \mathbb{R}^{n}\,\,|\,\,A(x)\neq \emptyset \},\,\,\text{%
graph}A:=\{(x,y)\,\,|\,\,y\in A(x)\};
\end{equation*}%
to simplify, we may identify $A$ to its $\text{graph}$. The operator $A$ is
said to be monotone if 
\begin{equation*}
\langle y_{1}-y_{2},x_{1}-x_{2}\rangle \geq 0\,\,\,\mbox{for all}%
\,\,\,(x_{i},y_{i})\in \text{graph}A,i=1,2.
\end{equation*}%
If, in addition, $A$ is not properly included in any other monotone
operator, then $A$ is said to be\ \textit{maximal monotone}. In this case,
for any $x\in \text{dom}A$, $A(x)$ is closed and convex; hence, $%
(A(x))^{\circ }$ is singleton. By the maximal property, if a sequence $%
(x_{n},y_{n})_{n}\subset A$ is such that $(x_{n},y_{n})\rightarrow (x,y)$ as 
$n\rightarrow +\infty ,$ then $(x,y)\in A.$

Take $f\in L^{1}(0,T;\mathbb{R}^{n})$ for $T>0$. The differential inclusion
given in $\mathbb{R}^{n}$ as 
\begin{equation*}
\dot{x}(t)\in f(t)-A(x(t))\,\,\text{a.e.}\,\,t\in \lbrack 0,T],x(0)=x_{0}\in 
\overline{\text{dom}A},
\end{equation*}%
always has a unique solution $x(\cdot ):=x(\cdot ;x_{0})$ (see \cite{Br}),
that satisfies for a.e. \ $t\in \lbrack 0,T]$ 
\begin{equation*}
\frac{d^{+}x(t)}{dt}:=\underset{t^{\prime }\downarrow t}{\lim }\frac{%
x(t^{\prime })-x(t)}{t^{\prime }-t}=f(t^{+})-\Pi _{A(x(t))}(f(t+0)),
\end{equation*}%
where $f(t^{+}):=\underset{h\rightarrow 0,h\neq 0}{\lim }\frac{1}{h}%
\int_{t}^{t+h}f(\tau )d\tau .$

Finally, we recall Gronwall's Lemma

\begin{lemma}
\label{Gronwall} \textrm{{(Gronwall's Lemma \cite{AHB2})} Let $T>0\ $ and $%
a,b\in L^{1}(t_{0},t_{0}+T;\mathbb{R})$ such that $b(t)\geq 0$ a.e. $t\in
\lbrack t_{0},t_{0}+T].$ If, for some $0\leq \alpha <1$, an absolutely
continuous function $w:[t_{0},t_{0}+T]\rightarrow \mathbb{R}_{+}$ satisfies 
\begin{equation*}
(1-\alpha )w^{\prime }(t)\leq a(t)w(t)+b(t)w^{\alpha }(t)\,\,\ \text{a.e.}%
\,\,t\in \lbrack t_{0},t_{0}+T],
\end{equation*}%
then 
\begin{equation*}
w^{1-\alpha }(t)\leq w^{1-\alpha }(t_{0})e^{\int_{t_{0}}^{t}a(\tau )d\tau
}+\int_{t_{0}}^{t}e^{\int_{s}^{t}a(\tau )d\tau }b(s)ds,\,\,\forall t\in
\lbrack t_{0},t_{0}+T].
\end{equation*}%
}
\end{lemma}

\section{Solutions of the system}

In this section, we investigate and review some properties of the solutions
of differential inclusion $(\ref{ieq})$, that is given by 
\begin{equation*}
\dot{x}(t)\in F(x(t))-A(x(t)),\,\,\mbox{\ a.e.\ }t\geq 0,\text{ \ }%
\,\,x(0)=x_{0}\in \overline{\text{dom}A},
\end{equation*}%
where $A:H\rightrightarrows H$ is a maximal monotone operator and $F$ is an $%
L$-Lipschitz Cusco mapping.

\begin{definition}
A continuous function $x:[0,\infty )\rightarrow \mathbb{R}^{n}$ is said to
be a solution of $(\ref{ieq})$ if it is absolutely continuous on every
compact subset of $(0,+\infty )$ and satisfies 
\begin{equation*}
\dot{x}(t)\in F(x(t))-A(x(t))\text{ \ }\mbox{ \ a.e.\ }t\geq 0,\text{ }%
x(0)=x_{0}\in \overline{\normalfont{\text{dom$A$}}}.
\end{equation*}
\end{definition}

The following characterization will be useful in the sequel.

\begin{proposition}
\label{pro3.1}A continuous function $x:[0,\infty )\rightarrow \mathbb{R}^{n}$
is a solution of \emph{(\ref{ieq})} iff $x(\cdot )$ is absolutely continuous
on every compact subset of $(0,+\infty ),$ and for every $T>0$ there exists
a function $f\in L^{\infty }(0,T;\mathbb{R}^{n})$ with $f(t)\in F(x(t))$
a.e. $t\in \lbrack 0,T],$ such that 
\begin{equation}
\dot{x}(t)\in f(t)-A(x(t))\,\,\text{a.e.}\,\,\,t\in \lbrack
0,T],\,\,x(0)=x_{0}\in \overline{\normalfont{\text{dom$A$}}}.  \label{301}
\end{equation}
\end{proposition}

\begin{proof}
The sufficient condition is clear and, so, we only need to justify the
necessary part. Suppose that $x(\cdot )$ is any solution of $(\ref{ieq})$
and fix $T>0.$ Since $F$ is Lipschitz and $x(\cdot )$ is continuous, there
exists $m>0$ such that $F(x(t))\subset \mathrm{B}_{m}$ for all $t\in \lbrack
0,T]$. We define the set-valued mapping $G:[0,T]\rightrightarrows \mathbb{R}%
^{n}$ as 
\begin{equation*}
G(t):=[\dot{x}(t)+A(x(t))]\cap F(x(t))=\Big[\lbrack \dot{x}(t)+A(x(t))]\cap 
\mathrm{B}_{m}\Big]\cap F(x(t)).
\end{equation*}%
We are going to check that $G$ is measurable. Since operator $A$ is maximal
monotone, the mappings 
\begin{equation*}
x\mapsto A_{n}(x):=A(x)\cap \mathrm{B}_{n},\text{ }n\geq 1,
\end{equation*}%
are upper semi-continuous, and so are the mappings 
\begin{equation*}
t\mapsto A_{n}(x(t)):=A(x(t))\cap \mathrm{B}_{n},\text{ }n\geq 1,
\end{equation*}%
due to the continuity of the solution $x(\cdot ).$ Then, due to the relation 
$A(x(t))=\cup _{n\in \mathbb{N}}{A_{n}(}x(t)),$ we deduce that the
multifunction $t\longmapsto A(x(t)\mathbb{)}$ is measurable. Since $\dot{x}%
(t)=\lim_{n\rightarrow +\infty }n(x(t+\frac{1}{n})-x(t))$ for a.e. $t\in
\lbrack 0,T],$ $\dot{x}(\cdot )$ is measurable, we deduce that the
multifunction $t\longmapsto \lbrack \dot{x}(t)+A(x(t))]\cap \mathrm{B}_{m}$
is measurable.\ Similarly, the multifunction $t\longmapsto F(x(t))$ is
measurable. Consequently, according to \cite[Proposition III.4]{CV}, the
mapping $G$ is measurable, and we conclude from \cite[Theorem III.6]{CV}
that $G$ admits a measurable selection; i.e., a measurable function $%
f:[0,T]\rightarrow \mathbb{R}^{n}$ such that 
\begin{equation*}
f(t)\in G(t)=[\dot{x}(t)+A(x(t))]\cap \mathrm{B}_{m}\cap F(x(t))\subset
F(x(t))\text{ }\mbox{\ a.e. \ }t\in \lbrack 0,T].
\end{equation*}%
Hence,\ $\dot{x}(t)\in f(t)-A(x(t))$ and $\left\Vert f(t)\right\Vert \leq
\left\Vert F(x(t))\right\Vert \leq m$, so that $f\in L^{\infty }(0,T;\mathbb{%
R}^{n}).$
\end{proof}

\bigskip

The next theorem shows that differential inclusion $(\ref{ieq})$ has at
least one solution whenever $x_{0}\in \overline{\text{dom}A}$. We use the
following lemma, which is a particular case of \cite[Theorem A]{A}.

\begin{lemma}
\label{l.3.1}Let $G:\mathbb{R}^{n}\rightrightarrows \mathbb{R}^{n}$ be a
Lipschitz multifunction\ with nonempty, convex and\ compact values, and let $%
x\in \mathbb{R}^{n},v\in G(x).$ Then there exists a Lipschitz selection $f$
of $G$ such that $f(x)=v.$
\end{lemma}

\begin{theorem}
\label{theo.3.1}Differential inclusion $(\ref{ieq})$ has at least one
solution.
\end{theorem}

\begin{proof}
Fix $x_{0}\in \overline{\text{dom}A}$ and, according to Lemma $\ref{l.3.1} $%
, let $f$ be a Lipschitz selection of $F.$ Then the differential inclusion 
\begin{equation*}
\dot{x}(t)\in f(x(t))-A(x(t)),\,\,\text{a.e.}\,\,t\geq 0,x(0)=x_{0},
\end{equation*}%
admits a unique solution $x(\cdot )$, which is absolutely continuous on
every compact subset of $(0,+\infty )$ (see e.g. \cite{Br, B}). It follows
that $x(\cdot )$ is also a solution of differential inclusion $(\ref{ieq}).$
\end{proof}

We also give some further properties of the solutions of differential
inclusion $(\ref{ieq}),$ which will be used in the sequel. Given a set $%
S\subset H$ and $x\in \text{dom}A$ we denote 
\begin{equation*}
(S-A(x))^{\circ }:=\bigcup\limits_{s\in S}(s-A(x))^{\circ }=\left\{ s-\Pi
_{A(x)}(s)\mid s\in S\right\} .
\end{equation*}

\begin{proposition}
\label{pro3.2}Fix\ $x_{0}\in \overline{\normalfont{\text{dom$A$}}}$ and let $%
x(\cdot ):=x(\cdot ;x_{0})$ be\ any solution of $(\ref{ieq})$. Then the\
following assertions hold\emph{:}

\begin{itemize}
\item[(i)] $x(t)\in \normalfont{\text{dom$A$}},$ for every\ $t>0,$ and for
a.e. $t\geq 0$ 
\begin{equation*}
\frac{d^{+}x(t)}{dt}:=\underset{h\downarrow 0}{\lim }\frac{x(t+h)-x(t)}{h}%
\in \left( F(x(t))-A(x(t))\right) ^{\circ }.
\end{equation*}

Conversely, if $x_{0}\in \normalfont{\text{dom$A$}}$, then for any $v\in
\lbrack F(x_{0})-A(x_{0})]^{\circ }$ there exists a solution $y(\cdot )$ of $%
(\ref{ieq})$ such that 
\begin{equation*}
y(0)=x_{0},\,\frac{d^{+}y(0)}{dt}=v.
\end{equation*}

\item[(ii)] There exists a real number $c>0$ such that for any $x_{0}\in %
\normalfont{\text{dom$A$}}$ and any solutions $x(\cdot ):=x(\cdot ;x_{0})$
and $y(\cdot ):=y(\cdot ;x_{0})$ of $(\ref{ieq})$, one has for all $t\geq 0$ 
\begin{equation*}
\left\Vert x(t)-x_{0}\right\Vert \leq 3(\left\Vert F(x_{0})\right\Vert
+\left\Vert A^{\circ }(x_{0})\right\Vert )te^{ct},
\end{equation*}%
\begin{equation*}
\left\Vert x(t)-y(t)\right\Vert \leq 4(\left\Vert F(x_{0})\right\Vert
+\left\Vert A^{\circ }(x_{0})\right\Vert )te^{ct}.
\end{equation*}%
Consequently, for every $T>0$ there exists $\rho >0$ such that 
\begin{equation*}
x(t)\in \mathrm{B}(x_{0},\rho )\,\,\forall t\in \lbrack 0,T].
\end{equation*}
\end{itemize}
\end{proposition}

\begin{proof}
$(i)$\ According\ to Proposition \ref{pro3.1}, for each $T>0$ there exists
some $f\in L^{\infty }(0,T;\mathbb{R}^{n})$ with $f(t)\in F(x(t))$ a.e. $%
t\in \lbrack 0,T],$ such that $x(\cdot )$ is the unique solution of $(\ref%
{301});$ hence, by \cite{Br} we deduce that $x(\cdot )$ satisfies\ $x(t)\in 
\text{dom}A$ for all $t\in \left( 0,T\right) ,$ and 
\begin{equation}
\frac{d^{+}x(t)}{dt}=\left( f(t^{+})-A(x(t))\right) ^{\circ }\text{ \ a.e. }%
t\in \left( 0,T\right) ,  \label{b}
\end{equation}%
where $f(t^{+}):=\lim_{h\rightarrow 0}h^{-1}\int_{0}^{h}f(t+\tau )d\tau .$
Moreover, given $\varepsilon >0$ there exists some $h>0$ such that for a.e. $%
\tau \in (0,h)$ we have 
\begin{equation*}
f(t+\tau )\in F(x(t+\tau ))\subset F(x(t))+L\left\Vert x(t+\tau
)-x(t)\right\Vert \mathbb{B}\subset F(x(t))+\varepsilon L\mathbb{B},
\end{equation*}%
and so $\underset{h\rightarrow 0^{+}}{\lim }\frac{1}{h}\int_{0}^{h}f(t+\tau
)d\tau \in F(x(t))+\varepsilon L\mathbb{B}$ (this last set is convex and
closed). Hence, as $\varepsilon $ goes to $0$ we get $f(t^{+})\in F(x(t))$,
and $(i)$ follows from (\ref{b}).

Conversely, we assume that $x_{0}\in \text{dom}A$ and take $v\in \lbrack
F(x_{0})-A(x_{0})]^{\circ }.$ We choose $w\in F(x_{0})$ such that $v=w-\Pi
_{A(x_{0})}(w).$ According to Lemma \ref{l.3.1}, there exists a Lipschitz
selection $f$ of $F$ such that $f(x_{0})=w$. Then the unique solution $%
y(\cdot )$ of the following differential inclusion 
\begin{equation*}
\dot{y}(t)\in f(y(t))-A(y(t)),\text{ \ }y(0)=x_{0},
\end{equation*}%
satisfies 
\begin{equation*}
\frac{d^{+}y(0)}{dt}=f(x_{0})-\Pi _{A(x_{0})}(f(x_{0}))=w-\Pi _{A(x_{0})}(w),
\end{equation*}%
and the proof of $(i)$ is complete.\newline
$(ii)$ Let $x(\cdot )$ be a solution of differential inclusion $(\ref{ieq})$%
, with $x(0)=x_{0},$ and fix $T>0$. Then by Proposition \ref{pro3.1} there
exist functions $k,g\in L^{1}(0,T;\mathbb{R}^{n})$ such that $k(t)\in
F(x(t)),\,g(t)\in A(x(t)),\,$and%
\begin{equation*}
\dot{x}(t)=k(t)-g(t)\text{ \ a.e}\,\,t\in \left[ 0,T\right] .
\end{equation*}%
We also choose by Lemma \ref{l.3.1} a Lipschitz mapping $f:\mathbb{R}%
^{n}\rightarrow \mathbb{R}^{n},$ with Lipschitz constant $c$ ($c\geq L$),
and consider the unique solution $z(\cdot )$ of the differential inclusion 
\begin{equation*}
\dot{z}(t)\in f(z(t))-A(z(t))\,\,\mbox{a.e.}\,\,t\geq 0,z(0)=x_{0}.
\end{equation*}%
So, for any $t\geq 0$ one has $\left\Vert \frac{d^{+}z(t)}{dt}\right\Vert
\leq \left\Vert \frac{d^{+}z(0)}{dt}\right\Vert $ and 
\begin{equation*}
\left\Vert \frac{d^{+}z(0)}{dt}\right\Vert =\left\Vert
(f(x_{0})-A(x_{0}))^{\circ }\right\Vert \leq \left\Vert F(x_{0})\right\Vert
+\left\Vert A^{\circ }(x_{0})\right\Vert ,
\end{equation*}%
so that 
\begin{eqnarray}
\left\Vert z(t)-x_{0}\right\Vert &\leq &\int_{0}^{t}e^{c\tau }\left\Vert 
\frac{d^{+}z(0)}{dt}\right\Vert d\tau =\frac{e^{ct}-1}{c}\left\Vert \frac{%
d^{+}z(0)}{dt}\right\Vert  \notag \\
&\leq &\frac{e^{ct}-1}{c}(\left\Vert F(x_{0})\right\Vert +\left\Vert
A^{\circ }(x_{0})\right\Vert )  \label{e} \\
&\leq &te^{ct}(\left\Vert F(x_{0})\right\Vert +\left\Vert A^{\circ
}(x_{0})\right\Vert ).  \label{302}
\end{eqnarray}%
By the Lipschitzianity of $F$ we choose a function $w(\cdot ):\left[ 0,T%
\right] \rightarrow \mathbb{R}^{n}$ such that 
\begin{equation}
w(t)\in F(z(t)),\text{ }\left\Vert k(t)-w(t)\right\Vert \leq L\left\Vert
x(t)-z(t)\right\Vert \ \text{\ }\forall t\in \left[ 0,T\right] .  \label{c}
\end{equation}%
Then we obtain%
\begin{equation*}
\begin{split}
\langle \dot{x}(t)-\dot{z}(t),x(t)-z(t)\rangle =& \left\langle
k(t)-g(t)-f(z(t))+\Pi _{A(z(t))}(f(z(t))),x(t)-z(t)\right\rangle \\
=& \left\langle k(t)-f(z(t)),x(t)-z(t)\right\rangle \\
&\,\,\, +\underset{\leq 0,\text{ by the monotonicity of }A}{\underbrace{%
\left\langle -g(t)+\Pi _{A(z(t))}(f(z(t))),x(t)-z(t)\right\rangle }} \\
\leq & \langle k(t)-w(t),x(t)-z(t)\rangle +\langle
w(t)-f(z(t)),x(t)-z(t)\rangle \\
\leq & L\left\Vert x(t)-z(t)\right\Vert ^{2}+2\left\Vert F(z(t))\right\Vert
\left\Vert x(t)-z(t)\right\Vert \text{ \ (by (\ref{c}))} \\
\leq & L\left\Vert x(t)-z(t)\right\Vert ^{2}+2\big(\left\Vert
F(x_{0})\right\Vert +L\left\Vert z(t)-x_{0}\right\Vert \big)\left\Vert
x(t)-z(t)\right\Vert \\
\leq & L\left\Vert x(t)-z(t)\right\Vert ^{2}+ \\
& \,\,\,2\Big(\left\Vert F(x_{0})\right\Vert +(e^{ct}-1)(\left\Vert
F(x_{0})\right\Vert +\left\Vert A^{\circ }(x_{0})\right\Vert )\Big)%
\left\Vert x(t)-z(t)\right\Vert \\
\leq & L\left\Vert x(t)-z(t)\right\Vert ^{2}+2(\left\Vert
F(x_{0})\right\Vert +\left\Vert A^{\circ }(x_{0})\right\Vert
)e^{ct}\left\Vert x(t)-z(t)\right\Vert .
\end{split}%
\end{equation*}%
Consequential, from\ the Gronwall Lemma we get,\ for every $t\geq 0$, 
\begin{equation*}
\left\Vert x(t)-z(t)\right\Vert \leq 2(\left\Vert F(x_{0})\right\Vert
+\left\Vert A^{\circ }(x_{0})\right\Vert )te^{ct},
\end{equation*}%
which together with $(\ref{302})$ give us 
\begin{equation*}
\left\Vert x(t)-x_{0}\right\Vert \leq 3(\left\Vert F(x_{0})\right\Vert
+\left\Vert A^{\circ }(x_{0})\right\Vert )te^{ct},
\end{equation*}%
and, for every other solution $y=y(\cdot ;x_{0}),$%
\begin{equation*}
\left\Vert x(t)-y(t)\right\Vert \leq \left\Vert x(t)-z(t)\right\Vert
+\left\Vert y(t)-z(t)\right\Vert \leq 4(\left\Vert F(x_{0})\right\Vert
+\left\Vert A^{\circ }(x_{0})\right\Vert )te^{ct};
\end{equation*}%
that is the conclusion of $(ii)$ follows.
\end{proof}

\section{Strong and weak invariant sets}

In this section, we give explicit characterizations for a closed set $%
S\subset \mathbb{R}^{n}$ to be\ strong or\ weak invariant for differential
inclusion $(\ref{ieq}),$ 
\begin{equation*}
\dot{x}(t)\in F(x(t))-A(x(t)),\,\,\mbox{\ a.e.\ }t\geq 0,\text{ \ }%
\,\,x(0)=x_{0}\in \overline{\text{dom}A},
\end{equation*}%
where $A:H\rightrightarrows H$ is a maximal monotone operator and $F$ is an $%
L$-Lipschitz Cusco mapping. Invariance criteria are written\ exclusively by
means of the data; that is, multifunction $F$ and operator $A$, and involve
the geometry of the set $S$, using the associated proximal and Fr\'echet
normal cones.

\begin{definition}
Let $S$ be a closed subset of $\mathbb{R}^n.$

\begin{itemize}
\item[(i)] $S$ is said to be strong invariant if for any $x_{0}\in S\cap 
\overline{\normalfont{\text{dom$A$}}}$ and any solution $x(\cdot ;x_{0})$ of 
$(\ref{ieq}) $, we have 
\begin{equation*}
x(t;x_{0})\in S\,\,\,\forall t\geq 0.
\end{equation*}

\item[(ii)] $S$ is said to be weak invariant if for any $x_{0}\in S\cap 
\overline{\normalfont{\text{dom$A$}}}$, there exists at least one solution $%
x(\cdot ;x_{0})$ of $(\ref{ieq})$ such that 
\begin{equation*}
x(t;x_{0})\in S\,\,\,\forall t\geq 0.
\end{equation*}
\end{itemize}
\end{definition}

Since any solution of differential inclusion $(\ref{ieq})$ lives in $%
\overline{\text{dom}A}$ (Proposition \ref{pro3.2}), we may assume without
loss of generality that $S$ is a closed subset of $\overline{\text{dom}A}$.
We shall need the following two lemmas.

\begin{lemma}
(e.g. \cite[Lemma A.1]{AHT1})\label{l.4.1}Let $S\subset \mathbb{R}^{n}$ be
closed. Then for every $x\in \mathbb{R}^{n}\setminus S$ we have 
\begin{equation*}
\partial _{L}d_{S}(\cdot )(x)\in \Big\{\frac{x-\Pi _{S}(x)}{d_{S}(x)}\Big\}\,%
\text{and }\partial _{C}d_{S}(\cdot )(x)\in \overline{\normalfont{\text{co}}}%
\Big\{\frac{x-\Pi _{S}(x)}{d_{S}(x)}\Big\}.
\end{equation*}
\end{lemma}

\begin{lemma}
\label{l.4.2}Let $\varphi :\mathbb{R}^{n}\rightarrow \mathbb{R}$ be an $l$%
-Lipschitz function. Then for every $x\in \mathbb{R}^{n}$ we have 
\begin{equation*}
\varphi (x+v)\leq \varphi (x)+\varphi ^{0}(x;v)+o(\left\Vert v\right\Vert ),%
\text{ }v\in \mathbb{R}^{n}.
\end{equation*}
\end{lemma}

\begin{proof}
We proceed by contradiction and suppose that for some\ $\alpha >0$ and
sequence $(v_{n})_{n}\subset \mathbb{R}^{n}\setminus \{\theta \}$ converging
to $\theta $ it holds 
\begin{equation}
\varphi (x+v_{n})-\varphi (x)>\varphi ^{0}(x;v_{n})+\alpha \left\Vert
v_{n}\right\Vert \text{ \ for all }n\geq 1.  \label{l.401}
\end{equation}%
Without loss of generality, we can assume that $\frac{v_{n}}{\left\Vert
v_{n}\right\Vert }\rightarrow v\neq \theta .$ Then 
\begin{equation*}
\begin{split}
\varphi (x+v_{n})-\varphi (x)=& \varphi \big(x+v_{n}-\left\Vert
v_{n}\right\Vert v+\left\Vert v_{n}\right\Vert v\big)-\varphi
(x+v_{n}-\left\Vert v_{n}\right\Vert v) \\
& +\varphi (x+v_{n}-\left\Vert v_{n}\right\Vert v)-\varphi (x) \\
\leq & \varphi \big(x+v_{n}-\left\Vert v_{n}\right\Vert v+\left\Vert
v_{n}\right\Vert v\big)-\varphi (x+v_{n}-\left\Vert v_{n}\right\Vert v) \\
& +l\left\Vert (v_{n}-\left\Vert v_{n}\right\Vert v)\right\Vert .
\end{split}%
\end{equation*}%
Hence, from inequality $(\ref{l.401})$ one gets\ 
\begin{equation*}
\frac{\varphi \big(x+v_{n}-\left\Vert v_{n}\right\Vert v+\left\Vert
v_{n}\right\Vert v\big)-\varphi (x+v_{n}-\left\Vert v_{n}\right\Vert v)}{%
\left\Vert v_{n}\right\Vert }+l\left\Vert \frac{v_{n}}{\left\Vert
v_{n}\right\Vert }-v\right\Vert \geq \varphi ^{0}(x;\frac{v_{n}}{\left\Vert
v_{n}\right\Vert })+\alpha ,
\end{equation*}%
which as $n\rightarrow \infty $ leads us to the contradiction $\varphi
^{0}(x;v)\geq \varphi ^{0}(x;v)+\alpha >\varphi ^{0}(x;v).$
\end{proof}

Before we state the main strong invariance theorem we give the following
result:

\begin{proposition}
\label{locally strong invariant}Let $S\subset \overline{\normalfont{%
\text{dom$A$}}}$ satisfy condition\ $(\ref{cond.}),$ and take\ $x_{0}\in S.$
If there is some $\rho >0$ such that for any $x\in \mathrm{B}(x_{0},\rho
)\cap S\cap \normalfont{\text{dom$A$}},$ 
\begin{equation}
\underset{\xi \in \mathrm{N}_{S}^{P}(x)}{\sup }\,\,\underset{v\in F(x)}{\sup 
}\,\,\underset{x^{\ast }\in A(x)}{\inf }\langle \xi ,v-x^{\ast }\rangle \leq
0,  \label{401}
\end{equation}%
then given\ any solution $x(\cdot ;x_{0})$ of $(\ref{ieq})$, there exists $%
T>0$ such that $x(t;x_{0})\in S$ for every $t\in \lbrack 0,T].$
\end{proposition}

\begin{proof}
Let $x(\cdot ):=x(\cdot ;x_{0})$ be any solution of differential inclusion $(%
\ref{ieq})$, so that for some $T_{1}>0$ we have 
\begin{equation}
x(t)\in \mathrm{B}(x_{0},\frac{\rho }{3})\cap \text{dom}A,\,\,\text{a.e.}%
\,\,t\in \lbrack 0,T_{1}],  \label{402b}
\end{equation}%
where $\rho >0$ is as in the current assumption, and so (by condition\ $(\ref%
{cond.})$) 
\begin{equation}
\Pi _{S}(x(t))\subset \mathrm{B}(x_{0},\frac{2}{3}\rho )\cap S\cap \text{dom}%
A\subset \mathrm{B}(x_{0},\rho )\cap S\cap \text{dom}A\text{ \ for a.e. }%
t\in (0,T_{1}].  \label{402c}
\end{equation}%
We denote the function $\eta :[0,T_{1}]\rightarrow \mathbb{R}$ as 
\begin{equation*}
\eta (t):=d_{S}^{2}(x(t)).
\end{equation*}%
Fix $\varepsilon >0.$ Since the function $d_{S}^{2}(\cdot )$ is Lipschitz on
each bounded set and $x(\cdot )$ is absolutely continuous on $\left[
\varepsilon ,T_{1}\right] $, function $\eta $ is also absolutely continuous
on $[\varepsilon ,T_{1}];$ hence, differentiable on a set $T_{0}\subset
\lbrack \varepsilon ,T_{1}]$ of full measure (we may also suppose that (\ref%
{402c}) holds for all $t\in T_{0}$). We pick $t\in T_{0}$ so that, according
to Lemma \ref{l.4.2}, for all $s>0$%
\begin{equation}
\begin{split}
d_{S}^{2}(x(t+s))& =d_{S}^{2}(x(t)+\dot{x}(t)s+o(s)) \\
& \leq d_{S}^{2}(x(t)+\dot{x}(t)s)+o(s) \\
& \leq \Big(d_{S}(x(t))+sd_{S}^{0}(x(t);\dot{x}(t))+o(s)\Big)^{2}+o(s) \\
& \leq d_{S}^{2}(x(t))+2d_{S}(x(t))d_{S}^{0}(x(t);\dot{x}(t)s)+o(s),
\end{split}
\label{402}
\end{equation}%
While by Lemma $\ref{l.4.1}$\ we have 
\begin{eqnarray}
d_{S}(x(t))d_{S}^{0}(x(t);\dot{x}(t)) &=&d_{S}(x(t))\underset{\xi \in
\partial _{C}d(x(t))}{\max }\langle \xi ,\dot{x}(t)\rangle  \label{402e} \\
&\leq &\underset{u\in \Pi _{S}(x(t))}{\max }\langle x(t)-u,\dot{x}(t)\rangle
.  \notag
\end{eqnarray}%
Let us write $\dot{x}(t)$ as $\dot{x}(t)=v-w$ for some $v\in F(x(t))$ and $%
w\in A(x(t))$, and fix $u\in \Pi _{S}(x(t))$ $(\subset \mathrm{B}(x_{0},\rho
)\cap S\cap \text{dom}A$ by (\ref{402c})$).$ By the Lipschitzianity of $F$
we choose some $v^{\prime }\in F(u)$ such that 
\begin{equation*}
\left\Vert v-v^{\prime }\right\Vert \leq L\left\Vert x(t)-u\right\Vert
=Ld_{S}(x(t)).
\end{equation*}%
Since $x(t)-u\in \mathrm{N}_{S}^{P}(u),$ by the current hypothesis of the
theorem there exist $w^{\prime }\in A(u)$ such that 
\begin{equation*}
\langle x(t)-u,v^{\prime }-w^{\prime }\rangle \leq 0,
\end{equation*}%
which in turn yields, due to the monotonicity of $A$, 
\begin{equation*}
\begin{split}
\langle x(t)-u,\dot{x}(t)\rangle & =\langle x(t)-u,v-w\rangle \\
& =\langle x(t)-u,v-v^{\prime }\rangle +\langle x(t)-u,v^{\prime }-w^{\prime
}\rangle \\
& \,\,\,\,+\langle x(t)-u,w^{\prime }-w\rangle \\
& \leq L\left\Vert x(t)-u\right\Vert ^{2}=Ld_{S}^{2}(x(t)).
\end{split}%
\end{equation*}%
Thus, continuing with $(\ref{402})$ and $(\ref{402e})$ we arrive at 
\begin{equation*}
\eta (t+s)\leq \eta (t)+2L\eta (t)s+o(\left\Vert s\right\Vert ),
\end{equation*}%
which implies that $\dot{\eta}(t)\leq 2L\eta (t).$ Hence, by the Gronwall
Lemma, we obtain that $\eta (t)\leq \eta (\varepsilon )e^{2L(t-\varepsilon
)} $ for all $t\in T_{0},$ or, equivalently, $\eta (t)\leq \eta (\varepsilon
)e^{2L(t-\varepsilon )}$ for all $t\in \lbrack \varepsilon ,T_{1}].$ Then,
as $\varepsilon $ goes to $0$ we conclude that $\eta (t)=0$ for all $t\in
\lbrack 0,T_{1}],$ which proves that $x(t)\in S$ for all $t\in \lbrack
0,T_{1}]$.
\end{proof}

We give the required characterization of strong invariant closed sets with
respect to differential inclusion $(\ref{ieq}).$

\begin{theorem}
\label{strong invariant}Let $S$ be a closed subset of $\overline{%
\normalfont{\text{dom}A}}$ satisfying\ relation\ $(\ref{cond.})$. Then the
following statements are equivalent, provided that $\mathrm{N}_{S}=\mathrm{N}%
_{S}^{P}$ or $\mathrm{N}_{S}^{F}$ and $\mathrm{T}_{S}=\mathrm{T}_{S}^{B},$
or $\mathrm{T}_{S}=\overline{\normalfont{\text{co}}}\mathrm{T}_{S}^{B}.$

(i) $S$ is strong invariant for differential inclusion $(\ref{ieq}).$

(ii) For every $x\in S\cap \normalfont{\text{dom$A$}}$, one has 
\begin{equation}
v-\Pi _{A(x)}(v)\in \mathrm{T}_{S}(x)\,\,\,\forall v\in F(x).
\label{c.si.401}
\end{equation}

(iii) For every $x\in S\cap \normalfont{\text{dom$A$}}$, one has 
\begin{equation}
\lbrack v-A(x)]\cap \mathrm{T}_{S}(x)\neq \emptyset \,\,\,\forall v\in F(x).
\label{c.si.402}
\end{equation}

(iv) For every $x\in S\cap \normalfont{\text{dom$A$}}$, one has 
\begin{equation}
\underset{\xi \in \mathrm{N}_{S}(x)}{\sup }\,\,\,\underset{v\in F(x)}{\sup }%
\langle \xi ,v-\Pi _{A(x)}(v)\rangle \leq 0.  \label{c.si.403}
\end{equation}

(v) For every $x\in S\cap \normalfont{\text{dom$A$}}$, one has 
\begin{equation}
\underset{\xi \in \mathrm{N}_{S}(x)}{\sup }\,\,\,\underset{v\in F(x)}{\sup }%
\,\,\,\underset{x^{\ast }\in A(x)}{\inf }\langle \xi ,v-x^{\ast }\rangle
\leq 0.  \label{c.si.404}
\end{equation}

(vi) For every $x\in S\cap \normalfont{\text{dom$A$}}$, one has 
\begin{equation}
\underset{\xi \in \mathrm{N}_{S}(x)}{\sup }\,\,\,\underset{v\in F(x)}{\sup }%
\,\,\,\underset{x^{\ast }\in A(x)\cap \mathrm{B}_{\left\Vert F(x)\right\Vert
+\left\Vert A^{\circ }(x)\right\Vert }}{\inf }\langle \xi ,v-x^{\ast
}\rangle \leq 0.  \label{t.si.405}
\end{equation}
\end{theorem}

\begin{proof}
The implication $(ii)\Rightarrow (iii)$ and $\ (vi)\Rightarrow (v)$ are
trivial, while the implications $(ii)\Rightarrow (iv)$ and $(iii)\Rightarrow
(v)$ come from the relation $\mathrm{T}_{S}(x)\subset (\mathrm{N}%
_{S}^{F}(x))^{\ast }$ for all $x\in S.$ The implications $(v)$ $($with $%
\mathrm{N}_{S}=\mathrm{N}_{S}^{P})\Rightarrow (i)$ is a direct consequence
of Proposition $\ref{locally strong invariant}$.

$(i)\Rightarrow (ii).$ To prove this implication we suppose that $S$ is
strong invariant and take $x_{0}\in S\cap \text{dom}A$ and $v\in F(x_{0}).$
According to Lemma \ref{l.3.1}, there exists a Lipschitz selection $f$ of $F$
such that $f(x_{0})=v,$ and so there is a unique solution $x(\cdot )$ of the
following differential inclusion, 
\begin{equation*}
\dot{x}(t)\in f(x(t))-A(x(t)),\,\,\text{a.e.}\,\,t\geq 0,\,\,x(0)=x_{0}.
\end{equation*}%
It follows that\ $x(\cdot )$ is also a solution of differential inclusion $(%
\ref{ieq}),$ so that $x(t)\in S$ for any $t\geq 0$. Then we get 
\begin{equation*}
v-\Pi _{A(x_{0})}(v)=(f(x_{0})-A(x_{0}))^{\circ }=\frac{d^{+}x(0)}{dt}=%
\underset{t\downarrow 0}{\lim }\frac{x(t)-x_{0}}{t}\in \mathrm{T}%
_{S}^{B}(x_{0})\subset \mathrm{T}_{S}(x_{0}).
\end{equation*}%
$(iv)\Rightarrow (vi).$ This implication holds since for any $x\in \text{dom}%
A$ and $v\in F(x)$ we\ have that 
\begin{equation*}
\begin{split}
\left\Vert \Pi _{A(x)}(v)\right\Vert & \leq \left\Vert \Pi
_{A(x)}(v)-A^{\circ }(x)\right\Vert +\left\Vert A^{\circ }(x)\right\Vert \\
& =\left\Vert \Pi _{A(x)}(v)-\Pi _{A(x)}(\theta )\right\Vert +\left\Vert
A^{\circ }(x)\right\Vert \\
& \leq \left\Vert v\right\Vert +\left\Vert A^{\circ }(x)\right\Vert \leq
\left\Vert F(x)\right\Vert +\left\Vert A^{\circ }(x)\right\Vert .
\end{split}%
\end{equation*}%
The proof of the theorem is complete.
\end{proof}

The following proposition, which provides the counterpart of Proposition $%
\ref{locally strong invariant}$ for the weak invariance, is essentially
given in \cite[Theorem 1]{DRW}. The specification of the interval on which
the solution remains in $S$ also comes from the proof given\ in that paper.

\begin{proposition}
\label{locally weak invariant}Let $S\subset \normalfont{\text{dom$A$}}$ be
closed and take\ $x_{0}\in S$ such that, for some\ $r,m>0$,\ 
\begin{equation}
||A^{\circ}(x)||\leq m\,\,\,\forall x\in S\cap \mathrm{B}(x_{0},r).
\label{t.lwi.401}
\end{equation}%
Assume that for all $x\in S\cap \mathrm{B}(x_{0},r)$, 
\begin{equation}
\underset{\xi \in \mathrm{N}_{S}(x)}{\sup }\,\,\underset{v\in F(x)}{\inf }%
\,\,\underset{x^{\ast }\in A(x)\cap \mathrm{B}_{m+\left\Vert F(x)\right\Vert
}}{\inf }\langle \xi ,v-x^{\ast }\rangle \leq 0.  \label{t.lwi.402}
\end{equation}%
Then there exists a solution $x(\cdot ;x_{0})$ of $(\ref{ieq})$ such that $%
x(t;x_{0})\in S$ for every $t\in \lbrack 0,T]$ with $T=\frac{r}{3}\Big(m+%
\underset{x\in \mathrm{B}(x_{0},r)\cap S}{\sup }\left\Vert F(x)\right\Vert %
\Big)^{-1}. $
\end{proposition}

Consequently, we obtain the desired characterization of weak invariant sets
with respect to differential inclusion $(\ref{ieq}).$ Recall that $A^{\circ
} $ is said to be locally bounded on $S$ if for every $x\in S$ we have 
\begin{equation}
\mathfrak{m}(x):=\underset{y\rightarrow x,y\in S}{\lim \sup }\left\Vert
A^{\circ }(y)\right\Vert <+\infty .  \label{mx0}
\end{equation}

\begin{theorem}
\label{weak invariant}Let $S\subset \normalfont{\text{dom$A$}}$ be a closed
set\ such that $A^{\circ}$ is locally bounded on $S$. Then the following
statements are equivalent provided that $\mathrm{T}_{S}$ and $\mathrm{N}_{S}$
are the same as the ones in Theorem $\ref{strong invariant}$\emph{:}

(i) $S$ is weak invariant for differential inclusion $(\ref{ieq}).$

(ii) For every $x\in S$, one has\ 
\begin{equation}
\cup _{v\in F(x)}\left[ v-A(x)\cap \mathrm{B}_{\mathfrak{m}(x)+\left\Vert
F(x)\right\Vert }\right] \cap \mathrm{T}_{S}(x)\neq \emptyset .
\label{c.wi.401}
\end{equation}

(iii) For every $x\in S$, one has 
\begin{equation}
\underset{\xi \in \mathrm{N}_{S}(x)}{\sup }\,\,\underset{v\in F(x)}{\inf }%
\,\,\underset{x^{\ast }\in A(x)\cap \mathrm{B}_{\mathfrak{m}(x)+\left\Vert
F(x)\right\Vert }}{\inf }\langle \xi ,v-x^{\ast }\rangle \leq 0.
\label{c.wi.402}
\end{equation}
\end{theorem}

\begin{proof}
$(i)\Rightarrow (ii)$. Given an $x_{0}\in S$ we choose a solution $x(\cdot
):=x(\cdot ;x_{0})$ of $(\ref{ieq})$ that belongs to $S$. Fix $\varepsilon
>0 $. By $(\ref{mx0})$ and the current assumption we also choose $\rho >0$
such that 
\begin{equation*}
\left\Vert A^{\circ }(x)\right\Vert \leq \mathfrak{m}(x_{0})+\varepsilon
\,\,\,\mbox{for all}\,\,\,x\in \mathrm{B}(x_{0},\rho )\cap S.
\end{equation*}%
Then for any $x\in \mathrm{B}(x_{0},\rho )\cap S$ and any $v\in F(x)$ we get 
\begin{equation*}
\left\Vert \Pi _{A(x)}(v)\right\Vert \leq \left\Vert \Pi _{A(x)}(v)-A^{\circ
}(x)\right\Vert +\left\Vert A^{\circ }(x)\right\Vert \leq \left\Vert
v\right\Vert +\left\Vert A^{\circ }(x)\right\Vert \leq \left\Vert
F(x)\right\Vert +\mathfrak{m}(x_{0})+\varepsilon .
\end{equation*}%
Let $T>0$ be such that $x(t)\in \mathrm{B}(x_{0},\rho )\cap S$ for all $t\in
\lbrack 0,T],$ so that for all $v\in F(x(t))$ and $t\in \lbrack 0,T]$ we
have\ 
\begin{equation*}
\left\Vert \Pi _{A(x(t))}(v)\right\Vert \leq \left\Vert F(x(t))\right\Vert +%
\mathfrak{m}(x_{0})+\varepsilon ;
\end{equation*}%
hence, by Proposition $\ref{pro3.2}(i)$, 
\begin{equation}
\dot{x}(t)\in F(x(t))-A(x(t))\cap \mathrm{B}_{\left\Vert F(x(t))\right\Vert +%
\mathfrak{m}(x_{0})+\varepsilon }\,\,\,a.e.\,\,\,t\in \lbrack 0,T],
\label{by}
\end{equation}%
and $x(\cdot )$ is Lipschitz on $[0,T]$ (observing that $\mathrm{B}%
_{\left\Vert F(x(t))\right\Vert +\mathfrak{m}(x_{0})+\varepsilon }\subset 
\mathrm{B}_{\left\Vert F(x_{0})\right\Vert +L\rho +\mathfrak{m}%
(x_{0})+\varepsilon }$)$.$ Take\ $w\in \text{Limsup}_{t\downarrow
0}t^{-1}(x(t)-x_{0})$ (this Painleve-Kuratowski upper limit is nonempty, due
to the Lipschitzianity of $x(\cdot )$). Then, since the mappings $x\mapsto
A(x)\cap \mathrm{B}_{\left\Vert F(x)\right\Vert +\mathfrak{m}%
(x_{0})+\varepsilon }$ and $x\mapsto F(x)$ are upper semicontinuous, by
using $(\ref{by})$ we get\ 
\begin{eqnarray}
w &\in &\text{Limsup}_{t\downarrow 0}\frac{1}{t}\int_{0}^{t}\dot{x}(\tau
)d\tau  \notag \\
&\subset &\text{Limsup}_{t\downarrow 0}\left( \overline{\text{co}}\left(
\bigcup\limits_{\tau \in \lbrack 0,t]}F(x(\tau ))-A(x(\tau ))\cap \mathrm{B}%
_{\left\Vert F(x(\tau ))\right\Vert +\mathfrak{m}(x_{0})+\varepsilon
}\right) \right)  \notag \\
&\subset &F(x_{0})-A(x_{0})\cap \mathrm{B}_{\left\Vert F(x_{0})\right\Vert +%
\mathfrak{m}(x_{0})+\varepsilon },  \label{be}
\end{eqnarray}%
and we conclude that, as $\varepsilon $ goes to $0$ (observe that $v$ is
independent of $\varepsilon $), 
\begin{equation*}
w\in F(x_{0})-A(x_{0})\cap \mathrm{B}_{\left\Vert F(x_{0})\right\Vert +%
\mathfrak{m}(x_{0})}.
\end{equation*}%
Thus, $(ii)$ follows, due to the obvious fact that $\text{Limsup}%
_{t\downarrow 0}t^{-1}(x(t)-x_{0})\subset \mathrm{T}_{S}(x_{0}).$

$(iii)\Rightarrow (i).$ Fix $x_{0}\in S$. By $(\ref{mx0})$ we choose\ $r,m>0$
such that $\mathfrak{m}(x)\leq m$ for every $x\in S\cap \mathrm{B}(x_{0},r).$
It suffices to prove that 
\begin{equation*}
\bar{T}:=\sup \{T:\exists \,\,x(\cdot ;x_{0})\,\,\text{a solution of $(\ref%
{ieq})$ such that}\,\,x(t;x_{0})\in S\,\,\,\forall t\in \lbrack
0,T]\}=+\infty .
\end{equation*}%
According to Proposition $\ref{locally weak invariant},$ there exist some $%
T_{1}>0$ and a solution $x_{1}(\cdot ;x_{0})$ of differential inclusion $(%
\ref{ieq})$ such that $x_{1}(t;x_{0})\in S$ for all $t\in \lbrack 0,T_{1}];$
hence, $\bar{T}\geq T_{1}>0.$

We proceed by contradiction and assume that $\bar{T}<+\infty .$ By
Proposition $\ref{pro3.2}$, we let $r_{1}>0$ be such that for every solution 
$x(\cdot ;x_{0})$ of $(\ref{ieq})$ we have 
\begin{equation*}
x(t;x_{0})\in \mathrm{B}(x_{0},r_{1})\,\,\forall t\in \lbrack 0,\bar{T}].
\end{equation*}%
We set 
\begin{equation*}
k:=\underset{x\in \mathrm{B}(x_{0},r_{1}+1)}{\sup }\left\Vert
F(x)\right\Vert +\underset{x\in \mathrm{B}(x_{0},r_{1}+1)\cap S}{\sup }%
\left\Vert A^{\circ }(x)\right\Vert ,
\end{equation*}%
so that $k<+\infty ,$ due to $(\ref{mx0})$ and the compactness of the set $%
\mathrm{B}(x_{0},r_{1}+1)\cap S.$ By definition of $\bar{T},$ for $%
0<\varepsilon <\min \left\{ \frac{1}{3k},\bar{T}\right\} $ we choose a
solution $x_{\varepsilon }(\cdot ;x_{0})$ of $(\ref{ieq})$ such that$\
x_{\varepsilon }(t;x_{0})\in S\ $for all $t\in \lbrack 0,\bar{T}-\varepsilon
].$ We put\ $y_{0}:=x_{\varepsilon }(\bar{T}-\varepsilon ;x_{0})\in \mathrm{B%
}(x_{0},r_{1})\cap S,$ so that $\mathrm{B}(y_{0},1)\subset \mathrm{B}%
(x_{0},r_{1}+1)$ and the following relations follows easily 
\begin{equation*}
||A^{\circ }(y)||\leq \underset{u\in \mathrm{B}(x_{0},r_{1}+1)\cap S}{\sup }%
\left\Vert A^{\circ }(u)\right\Vert =:m_{1}\,\,\,\forall y\in S\cap \mathrm{B%
}(y_{0},1),
\end{equation*}%
\begin{equation*}
\underset{\xi \in \mathrm{N}_{S}(y)}{\sup }\,\,\underset{v\in F(y)}{\inf }%
\,\,\underset{x^{\ast }\in A(y)\cap \mathrm{B}_{m_{1}+\left\Vert
F(y)\right\Vert }}{\inf }\langle \xi ,v-x^{\ast }\rangle \leq 0\text{ for
all }y\in S\cap \mathrm{B}(y_{0},1).
\end{equation*}%
Then, according to Proposition $\ref{locally weak invariant}$, there exists
a solution $x_{2}(\cdot ;y_{0})$ of $(\ref{ieq})$ such that$%
\,\,x_{2}(t;y_{0})\in S$ for all $t\in \lbrack 0,\frac{1}{3k}].$
Consequently, the function $z(\cdot ;x_{0})$ defined as 
\begin{equation*}
z(t;x_{0}):=%
\begin{cases}
x_{\varepsilon }(t;x_{0}) & \mbox{\ if }s\in \lbrack 0,\bar{T}-\varepsilon ]
\\ 
x_{2}(t-\bar{T}+\varepsilon ;y_{0}) & \mbox{\ if }s\in \lbrack \bar{T}%
-\varepsilon ,+\infty \lbrack ,%
\end{cases}%
\end{equation*}%
is a solution of $(\ref{ieq})$ and satisfies $z(t;x_{0})\in S$ for all\ $%
t\in \lbrack 0,\tilde{T}]$ with $\tilde{T}:=\bar{T}+\frac{1}{3k}-\varepsilon
>\bar{T}$, which contradicts the definition of $\bar{T}.$ Hence $\bar{T}%
=\infty $, and\ $S$ is weak invariant.
\end{proof}

\section{Strong $a$-Lyapunov pairs}

In this section, we use the invariance results of the previous section to
characterize strong $a$-Lyapunov pairs with respect to differential
inclusion (\ref{ieq}), 
\begin{equation*}
\dot{x}(t)\in F(x(t))-A(x(t)),\,\,\mbox{\ a.e.\ }t\geq 0,\text{ \ }%
\,\,x(0)=x_{0}\in \overline{\text{dom}A},
\end{equation*}%
where $A:H\rightrightarrows H$ is a maximal monotone operator and $F$ is an $%
L$-Lipschitz Cusco mapping.

\begin{definition}
\label{d.Lyapunov}Let $V,W:\mathbb{R}^{n}\rightarrow \mathbb{R\cup \{}%
+\infty \mathbb{\}}$ be lsc functions such that $W\geq 0$ and let $a\geq 0$.
We say that $(V,W)$ is a strong\ $a$-Lyapunov pair for $(\ref{ieq})$ if for
any $x_{0}\in \overline{\normalfont{\text{dom$A$}}}$ we have 
\begin{equation}
e^{at}V(x(t;x_{0}))+\int_{0}^{t}W(x(\tau ;x_{0}))d\tau \leq V(x_{0})\,\,%
\text{\ \ \ \ \ \ \ }\forall t\geq 0,  \label{d.L.501}
\end{equation}%
for every\ solution $x(\cdot ;x_{0})$ of $(\ref{ieq})$.
\end{definition}

The following lemma shows that the non-regularity of the functions $V,W$
candidates to form $a$-Lyapunov pairs is mainly carried by the function $V.$
For $k\geq 1$ we denote 
\begin{equation}
W_{k}(x):=\underset{z\in \mathbb{R}^{n}}{\inf }\{W(z)+k\left\Vert
x-z\right\Vert \}.  \label{w-e}
\end{equation}

\begin{lemma}
\label{l5.3}Given a function $W:\mathbb{R}^{n}\rightarrow \mathbb{R}_{+}%
\mathbb{\cup \{}+\infty \mathbb{\}}$, $W_{k}$ defined in $(\ref{w-e})$\ is $%
k $-Lipschitz, and we have $W_{k}(x)\nearrow W(x)$ for all $x\in \mathbb{R}%
^{n} $. Moreover, if $x(\cdot ;x_{0})$ is a solution of\ differential
inclusion $(\ref{ieq})$, then $W$ satisfies inequality $(\ref{d.L.501})$ iff 
$W_{k}$ does for all $k\geq 1.$
\end{lemma}

\begin{proof}
The first statement of the lemma is known (see, e.g., \cite{CLSW}), and the
second statement of the lemma follows easily from Fatou's lemma.
\end{proof}

\begin{lemma}
\label{lt}Consider the operator $\hat{A}:\mathbb{R}^{n}\times \mathbb{R}%
^{3}\rightarrow \mathbb{R}^{n+3}$ and the function $\tilde{V}:\mathbb{R}%
^{n+1}\times \mathbb{R}_{+}\rightarrow \mathbb{R\cup \{+\infty \}}$ defined
as 
\begin{equation}
\hat{A}(x,\alpha ,\beta ,\gamma ):=(A(x),\theta _{\mathbb{R}^{3}}),\,\,%
\tilde{V}(x,\alpha ,\beta ):=e^{a\beta }V(x)+\alpha ,  \label{tranform}
\end{equation}%
together with the mappings $\hat{F}_{k}:\mathbb{R}^{n+3}\rightarrow \mathbb{R%
}^{n+3},$ $k\geq 1,$ given by (recall $(\ref{w-e})$) 
\begin{equation*}
\hat{F}_{k}(x,\alpha ,\beta ,\gamma ):=(F(x),W_{k}(x),1,0).
\end{equation*}%
Then $\hat{A}$ is maximal monotone with $\normalfont{\text{dom}}\hat{A}=%
\normalfont{\text{dom$A$}}\times \mathbb{R}^{3}$, $\hat{F}_{k}$ is Lipschitz
with constant $(L^{2}+k^{2})^{\frac{1}{2}}, $ and consequently, the
following differential inclusion possesses solutions,%
\begin{equation}
\dot{z}(t)\in \hat{F}_{k}(z(t))-\hat{A}(z(t))\,\,\text{a.e.}\,\,\,t\geq 0,%
\text{ }z(0)=z_{0}=(x_{0},y_{0},z_{0},w_{0})\in \overline{%
\normalfont{\text{dom$A$}}}\times \mathbb{R}^{3},  \label{ieq.transform}
\end{equation}%
and every solutions is written as 
\begin{equation*}
z(t;z_{0})=(x(t;x_{0}),y_{0}+\int_{0}^{t}W_{k}(x(\tau ;x_{0}))d\tau
,z_{0}+t,w_{0}),
\end{equation*}%
for a solution $x(\cdot ;x_{0})$ of $(\ref{ieq})$.
\end{lemma}

We need the following result which provides us with a local criterion for
strong $a$-Lyapunov pairs.

\begin{proposition}
\label{t5.3}Let $V,W:\mathbb{R}^{n}\rightarrow \mathbb{R\cup \{+\infty \}}$
be two proper lsc functions such that $\normalfont{\text{dom$V$}}\subset %
\normalfont{\text{dom$A$}},W\geq 0$ and let $a\geq 0.$ Fix\ $x_{0}\in %
\normalfont{\text{dom$V$}}$ and assume that\ for some $\rho >0$ we have, for
all\ $x\in \mathrm{B}(x_{0},\rho ),$ 
\begin{equation}
\underset{\xi \in \partial _{P}V(x)}{\sup }\,\,\underset{v\in F(x)}{\sup }%
\,\,\underset{x^{\ast }\in A(x)}{\inf }\langle \xi ,v-x^{\ast }\rangle
+aV(x)+W(x)\leq 0,  \label{sl.2}
\end{equation}%
\begin{equation}
\underset{\xi \in \partial _{P,\infty }V(x)}{\sup }\,\,\underset{v\in F(x)}{%
\sup }\,\,\underset{x^{\ast }\in A(x)}{\inf }\langle \xi ,v-x^{\ast }\rangle
\leq 0.  \label{sl.2b}
\end{equation}%
Then there exists some $T>0$ such that\ for every solution $x(\cdot ;x_{0})$
of differential inclusion $(\ref{ieq})$ one has 
\begin{equation*}
e^{at}V(x(t;x_{0}))+\int_{0}^{t}W(x(\tau ;x_{0}))d\tau \leq
V(x_{0})\,\,\forall t\in \lbrack 0,T].
\end{equation*}
\end{proposition}

\begin{proof}
First, by Proposition \ref{pro3.2}(ii) we let $c>0$ be such that for any
solutions $x(\cdot ):=x(\cdot ;x_{0})$ of $(\ref{ieq})$ it holds 
\begin{equation*}
\left\Vert x(t)-x_{0}\right\Vert \leq 3(\left\Vert F(x_{0})\right\Vert
+\left\Vert A^{\circ }(x_{0})\right\Vert )te^{ct}\text{ \ for all }t\geq 0,
\end{equation*}%
and choose $T>0$ such that 
\begin{equation}
3(\left\Vert F(x_{0})\right\Vert +\left\Vert A^{\circ }(x_{0})\right\Vert
)Te^{cT}\leq \rho .  \label{t}
\end{equation}%
As in Lemma \ref{lt}, we define the proper and lsc function $\tilde{V}:%
\mathbb{R}^{n+1}\times \mathbb{R}_{+}\rightarrow \mathbb{R\cup \{+\infty \}}$
as $\tilde{V}(x,\alpha ,\beta ):=e^{a\beta }V(x)+\alpha ,$ so that epi$%
\tilde{V}$ is closed and satisfies 
\begin{equation*}
\text{epi}\tilde{V}\subset \text{dom}V\times \mathbb{R}^{3}\subset \text{dom}%
A\times \mathbb{R}^{3}=\text{dom}\hat{A},
\end{equation*}%
where $\hat{A}$ is also defined as in\ Lemma \ref{lt}; hence, condition\ (%
\ref{cond.}) is obviously satisfied for epi$\tilde{V}$.

\textbf{Claim.} We claim that for any given $\tilde{z}%
:=(x_{1},y_{1},z_{1},w_{1})\in \text{epi}\tilde{V}$ with\ $\left\Vert
x_{1}-x_{0}\right\Vert <\rho $, there exists small enough $\varepsilon >0$
such that for each $(x,y,z,w)\in \mathrm{B}(\tilde{z},\varepsilon )\,\cap \,$%
epi$\tilde{V},$ $(\tilde{\xi},-\kappa )\in \mathrm{N}_{\text{epi}\tilde{V}%
}^{P}(x,y,z,w),\ $and $(v,W_{k}(x),1,0)\in \hat{F}_{k}(x,y,z,w)$ there
exists $x^{\ast }\in A(x)$ such that 
\begin{equation}
\langle (\tilde{\xi},-\kappa ),(v-x^{\ast },W_{k}(x),1,0)\rangle \leq 0.
\label{cl2}
\end{equation}

Indeed, with $\tilde{z}\ $as in the claim let us\ choose\ $\varepsilon >0$
such that 
\begin{equation*}
(x,y,z,w)\in \mathrm{B}(\tilde{z},\varepsilon )\cap \text{epi}\tilde{V}%
\Rightarrow x\in \mathrm{B}(x_{0},\rho ).
\end{equation*}%
Let $(x,y,z,w),$ $(\tilde{\xi},-\kappa ),$ and $(v,W_{k}(x),1,0)$ be as in
the claim, so that $x\in \mathrm{B}(x_{0},\rho )\cap \text{dom}V$ and $v\in
F(x)$, as well as\ $\kappa \geq 0$ (see \cite[Exercise 2.1]{CLSW}). We may
distinguish two cases:

$(i)$ If $\kappa >0$, then $w=\tilde{V}(x,y,z)$ and, without loss of
generality, we may\ suppose that $\kappa =1.$\ Hence, $\tilde{\xi}%
=(e^{az}\xi ,1,ae^{az}V(x))\in \partial _{P}\tilde{V}(x,y,z)$ for some\ $\xi
\in \partial _{P}V(x).$ Consequently, by the current hypothesis there exists 
$x^{\ast }\in A(x)$ such that 
\begin{equation*}
\langle \xi ,v-x^{\ast }\rangle +aV(x)+W_{k}(x)\leq \langle \xi ,v-x^{\ast
}\rangle +aV(x)+W(x)\leq 0.
\end{equation*}%
In other words, we have $(v-x^{\ast },W_{k}(x),1,0)\in \hat{F}_{k}(x,y,z,w)-%
\hat{A}(x,y,z,w)$ and%
\begin{eqnarray}
\langle (\tilde{\xi},-1),(v-x^{\ast },W_{k}(x),1,0)\rangle &=&\langle
(e^{az}\xi ,1,ae^{az}V(x),-1),(v-x^{\ast },W_{k}(x),1,0)\rangle  \notag \\
&=&e^{az}\langle \xi ,v-x^{\ast }\rangle +W_{k}(x)+ae^{az}V(x)  \notag \\
&=&e^{az}(\langle \xi ,v-x^{\ast }\rangle +aV(x)+W_{k}(x))  \notag \\
&&\quad \quad +(1-e^{az})W_{k}(x)\leq 0,  \label{t.sl.1}
\end{eqnarray}%
and (\ref{cl2}) follows.

$(ii)$ If $\kappa =0$, then $\tilde{\xi}\in \partial _{P,\infty }\tilde{V}%
(x,y,z)$ and, so, $(\tilde{\xi},-\kappa )=(\xi ,\theta _{\mathbb{R}^{3}})$
for some $\xi \in \partial _{P,\infty }V(x).$ Then, by arguing as in the
paragraph above, the current hypothesis yields some\ $x^{\ast }\in A(x)$
such that $\langle \xi ,v-x^{\ast }\rangle \leq 0.$ Hence, $(v-x^{\ast
},W_{k}(x),1,0)\in \hat{F}_{k}(x,y,z,w)-\hat{A}(x,y,z,w)$ and 
\begin{equation}
\langle (\tilde{\xi},0),(v-x^{\ast },W_{k}(x),1,0)\rangle =\langle \xi
,v-x^{\ast }\rangle \leq 0;  \label{t.sl.2}
\end{equation}%
that is, (\ref{cl2}) follows in this case too. The claim is proved.

Now, we take a solution $x(\cdot ;x_{0})$ of (\ref{ieq}), so that 
\begin{equation*}
z(\cdot ;z_{0}):=(x(\cdot ;x_{0}),\int_{0}^{\cdot }W_{k}(x(\tau
;x_{0}))d\tau ,\cdot ,V(x_{0})),
\end{equation*}
with $z_{0}:=(x_{0},0,0,V(x_{0})),$ becomes a solution of (\ref%
{ieq.transform}). Then, from the claim (with $\tilde{z}:=z_{0}$) above and
Proposition \ref{locally strong invariant}, there exists some\ $\bar{t}>0$
such that 
\begin{equation}
z(t;z_{0})\in \text{epi}\tilde{V}\,\,\,\forall t\in \lbrack 0,\bar{t}];
\label{f}
\end{equation}%
that is, 
\begin{equation}
\bar{T}:=\sup \{t\geq 0:\text{such that}\,z(s;z_{0})\in \text{epi}\tilde{V}%
\,\,\forall s\in \lbrack 0,t]\}>0.  \label{ee}
\end{equation}%
Let us show that $\bar{T}\geq T,$ where $T$ is defined in (\ref{t}). We
proceed by contradiction and assume that $\bar{T}<T$. Then, because (by
Proposition \ref{pro3.2}(ii)) 
\begin{equation*}
\left\Vert x(\bar{T};x_{0})-x_{0}\right\Vert \leq 3(\left\Vert
F(x_{0})\right\Vert +\left\Vert A^{\circ }(x_{0})\right\Vert )\bar{T}e^{c%
\bar{T}}<\rho ,
\end{equation*}%
and\ $z(\bar{T};z_{0})=(x(\bar{T};x_{0}),\int_{0}^{\bar{T}}W_{k}(x(\tau
;x_{0}))d\tau ,\bar{T},V(x_{0}))\in \text{epi}\tilde{V}$, from the claim
above (with $\tilde{z}:=z(\bar{T};z_{0})$) and Proposition \ref{locally
strong invariant}, there exists some $t_{1}>0$ such that $z(t;z(\bar{T}%
;z_{0}))\in \text{epi}\tilde{V}$ for all $t\in \lbrack 0,t_{1}]$. Thus, $z(t+%
\bar{T};z_{0})=z(t;z(\bar{T};z_{0}))\in \text{epi}\tilde{V}$ for every $t\in
\lbrack 0,t_{1}]$, and we get a contradiction to the definition of $\bar{T}$.

Finally, from (\ref{ee}) we get \ 
\begin{equation*}
e^{at}V(x(t;x_{0}))+\int_{0}^{t}W_{k}(x(\tau ;x_{0}))d\tau \leq
V(x_{0})\,\,\forall t\in \lbrack 0,T].
\end{equation*}%
Moreover, because $T$ is independent of $k,$ by taking the limit as $%
k\rightarrow \infty $ we arrive at (as $W_{k}(x)\nearrow W(x)$, by Lemma \ref%
{l5.3}) 
\begin{equation*}
e^{at}V(x(t;x_{0}))+\int_{0}^{t}W(x(\tau ;x_{0}))d\tau \leq
V(x_{0})\,\,\forall t\in \lbrack 0,T],
\end{equation*}%
which is the desired inequality.
\end{proof}

We give now the desired\ characterization of\ strong $a$-Lyapunov pairs.

\begin{theorem}
\label{StrongLyap}Let $V,W,$ and $a$ be as in Proposition $\ref{t5.3},$ and
let $\partial $ stand for either $\partial _{P}\ $or $\partial _{F}.$ Then
the pair\ $(V,W)$ is a strong $a$-Lyapunov pair for $(\ref{ieq})$ iff for
all\ $x\in \normalfont{\text{dom$V$}}$ 
\begin{equation}
\underset{\xi \in \partial V(x)}{\sup }\,\,\underset{v\in F(x)}{\sup }\,\,%
\underset{x^{\ast }\in A(x)}{\inf }\langle \xi ,v-x^{\ast }\rangle
+aV(x)+W(x)\leq 0,  \label{s1}
\end{equation}%
\begin{equation}
\underset{\xi \in \partial _{P,\infty }V(x)}{\sup }\,\,\underset{v\in F(x)}{%
\sup }\,\,\underset{x^{\ast }\in A(x)}{\inf }\langle \xi ,v-x^{\ast }\rangle
\leq 0.  \label{s2}
\end{equation}
\end{theorem}

\begin{proof}
To prove\ the sufficiency part, we take\ $x_{0}\in \normalfont{\text{dom$V$}}
$ and a solution $x(\cdot ;x_{0})$ of differential inclusion $(\ref{ieq})$.
By\ Proposition \ref{t5.3} there exists some $T>0$ such that 
\begin{equation}
e^{at}V(x(t;x_{0}))+\int_{0}^{t}W(x(\tau ;x_{0}))d\tau \leq
V(x_{0})\,\,\forall t\in \lbrack 0,T].  \label{ie}
\end{equation}%
It suffices to prove that the following quantity is $+\infty ,$%
\begin{equation*}
T:=\sup \{s\geq 0:(\ref{ie})\text{$\,$holds$\,\forall t\in \lbrack 0,s]$}\}.
\end{equation*}%
Otherwise, if $T$ is finite, then $x(T;x_{0})\in \normalfont{\text{dom$V$}}$
(because $V$ is lsc), and again from Proposition \ref{t5.3} we find $\eta >0$
such that for all $t\in \lbrack 0,\eta ]$, using\ the semi-group property of 
$x(\cdot ;x_{0}),$%
\begin{eqnarray*}
&&e^{a(t+T)}V(x(t+T;x_{0}))+\int_{0}^{t+T}W(x(\tau ;x_{0}))d\tau \\
&\leq &e^{aT}\left( e^{at}V(x(t+T;x_{0}))+\int_{T}^{t+T}W(x(\tau
;x_{0}))d\tau \right) +\int_{0}^{T}W(x(\tau ;x_{0}))d\tau \\
&\leq &e^{aT}V(x(T;x_{0}))+\int_{0}^{T}W(x(\tau ;x_{0}))d\tau \leq V(x_{0}),
\end{eqnarray*}%
and we get the contradiction $T\geq T+\eta $. Hence, $T=+\infty $ and $(\ref%
{ie})\,$holds for all $t\geq 0,$ showing that $(V,W)$ forms a strong
Lyapunov pair for differential inclusion $(\ref{ieq}).$

To prove the necessity of the current conditions, we start by verifying\ $(%
\ref{s1})\,$with $\partial =\partial _{F}$. We\ fix $x_{0}\in %
\normalfont{\text{dom$V$}}$ ($\subset \normalfont{\text{dom$A$}})$ and $v\in
F(x_{0}),$\ and, according to Proposition \ref{pro3.2}, we choose\ a
solution $x(\cdot ;x_{0})$ of differential inclusion (\ref{ieq}) such that\ $%
\frac{d^{+}x(0;x_{0})}{dt}=v-\Pi _{A(x_{0})}(v).$ Thus, since $(V,W)$ is
assumed to be a strong $a$-Lyapunov pair for (\ref{ieq}), we obtain for
every $t>0$ 
\begin{equation*}
\frac{V(x(t;x_{0}))-V(x_{0})}{t}+\frac{e^{at}-1}{t}V(x(t;x_{0}))+\frac{1}{t}%
\int_{0}^{t}W(x(\tau ;x_{0}))d\tau \leq 0,
\end{equation*}%
which give us, as $t\downarrow 0$, 
\begin{eqnarray*}
\sigma _{\partial _{F}V(x_{0})}(v-\Pi _{A(x_{0})}(v)) &\leq &V^{\prime
}(x_{0};v-\Pi _{A(x_{0})}(v)) \\
&\leq &\underset{t\downarrow 0}{\liminf }\frac{V(x(t;x_{0}))-V(x_{0})}{t}%
\leq -aV(x_{0})-W(x_{0}).
\end{eqnarray*}%
Hence, $(\ref{s1})$ follows with either $\partial =\partial _{F}$ or $%
\partial =\partial _{P}.$ To verify $(\ref{s2})$ we fix $x_{0}\in %
\normalfont{\text{dom$V$}},$ $v\in F(x_{0})$ and $\xi \in \partial
_{P,\infty }V(x_{0});$ that is, $(\xi ,0)\in \mathrm{N}_{\text{epi}%
V}^{P}(x_{0},V(x_{0})).$ According to Proposition \ref{pro3.2}, we choose\ a
solution $x(\cdot ;x_{0})$ of differential inclusion (\ref{ieq}) such that\ $%
\frac{d^{+}x(0;x_{0})}{dt}=v-\Pi _{A(x_{0})}(v).$ Since $(V,W)$ is strong $a 
$-Lyapunov for differential inclusion (\ref{ieq}), one has that $%
(x(t;x_{0}),e^{-at}V(x_{0}))\in $epi$V$ for all $t\geq 0.$ Then, by the
definition of the proximal normal cone, there exists $\eta >0$ such that for
all small $t\geq 0$ 
\begin{equation*}
\langle (\xi ,0),(x(t;x_{0}),e^{-at}V(x_{0}))-(x_{0},V(x_{0}))\rangle \leq
\eta \big(\left\Vert x(t;x_{0})-x_{0}\right\Vert
^{2}+|e^{-at}V(x_{0})-V(x_{0})|\big)^{2},
\end{equation*}%
and so 
\begin{equation*}
\langle \xi ,x(t;x_{0})-x_{0}\rangle \leq \eta \big(\left\Vert
x(t;x_{0})-x_{0}\right\Vert ^{2}+(e^{-at}-1)^{2}|V(x_{0})|^{2}\big).
\end{equation*}%
Hence, by dividing on\ $t>0$ and taking limits as $t\downarrow 0$, we obtain
that\ $\langle \xi ,v-\Pi _{A(x_{0})}(v)\rangle \leq 0,$ as we wanted to
prove.
\end{proof}

We give in the following corollary other criteria for strong $a$-Lyapunov
pairs for $(\ref{ieq}).$ Recall that $A^{\circ }$ is said to be locally
bounded on $\normalfont{\text{dom$V$}}$ if condition (\ref{mx0}) holds for
all $x\in \text{dom}V;$ that is, for every $x\in \normalfont{\text{dom$V$}}$
we have 
\begin{equation*}
\mathfrak{m}(x)=\underset{y\rightarrow x,y\in \normalfont{\text{dom$V$}}}{%
\lim \sup }\left\Vert A^{\circ }(y)\right\Vert <+\infty .
\end{equation*}%
We also observe that the function $\mathfrak{m}$ is upper semicontinuous at
every $x\in \mathbb{R}^{n}$ such that $\mathfrak{m}(x)<+\infty ;$ that is, 
\begin{equation}
\underset{y\rightarrow x,y\in \normalfont{\text{dom$V$}}}{\lim \sup }%
\mathfrak{m}(y)=\mathfrak{m}(x).  \label{uscm}
\end{equation}

\begin{corollary}
\label{c5.3}Let $V,W,$ and $a$ be as in Proposition $\ref{t5.3},$ and let $%
\partial $ stand for either $\partial _{P},$ $\partial _{F},$ or $\partial
_{L}.$ If $A^{\circ }$ is locally bounded on $\normalfont{\text{dom$V$}},$
then $(V,W)$ is a strong $a$-Lyapunov pair for $(\ref{ieq})$ iff one of the
following statements holds.

\begin{itemize}
\item[(i)] For any $x\in \normalfont{\text{dom$V$}}$, 
\begin{equation*}
\underset{\xi \in \partial V(x)}{\sup }\,\,\underset{v\in F(x)}{\sup }\,\,%
\underset{x^{\ast }\in A(x)\cap \mathrm{B}_{\left\Vert F(x)\right\Vert +%
\mathfrak{m}(x)}}{\inf }\langle \xi ,v-x^{\ast }\rangle +aV(x)+W(x)\leq 0.
\end{equation*}

\item[(ii)] For any $x\in \normalfont{\text{dom$V$}}$, 
\begin{equation*}
\underset{v\in F(x)}{\sup }V^{\prime }(x;v-\Pi _{A(x)}(v))+aV(x)+W(x)\leq 0.
\end{equation*}

\item[(iii)] For any $x\in \normalfont{\text{dom$V$}}$, 
\begin{equation*}
\underset{v\in F(x)}{\sup }\,\,\underset{x^{\ast }\in A(x)\cap \mathrm{B}%
_{\left\Vert F(x)\right\Vert +\mathfrak{m}(x)}}{\inf }V^{\prime
}(x;v-x^{\ast })+aV(x)+W(x)\leq 0.
\end{equation*}
\end{itemize}
\end{corollary}

\begin{proof}
$(ii)\Rightarrow (iii).$ This implication follows since that\ for any $x\in 
\text{dom}V$ $(\subset $dom$A)$ any $v\in F(x)$\ 
\begin{equation*}
\Vert \Pi _{A(x)}(v)\Vert \leq \left\Vert A^{\circ }(x)\right\Vert
+\left\Vert \Pi _{A(x)}(v)-A^{\circ }(x)\right\Vert \leq \left\Vert A^{\circ
}(x)\right\Vert +\left\Vert v\right\Vert \leq \mathfrak{m}(x)+\left\Vert
F(x)\right\Vert .
\end{equation*}%
$(iii)\Rightarrow (i).$ When $\partial $ stands for either $\partial _{P}$
or $\partial _{F}$ this implication follows from the relation\ $\sigma
_{\partial _{P}V(x)}(\cdot )\leq \sigma _{\partial _{F}V(x)}(\cdot )\leq
V^{\prime }(x;\cdot )$. If $\partial =\partial _{L},$\ we take $\xi \in
\partial _{L}V(x)$ and $v\in F(x),$ and choose\ sequences $(x_{i})$ and $%
(\xi _{i})$ such that 
\begin{equation*}
x_{i}\overset{V}{\rightarrow }x,\,\,\xi _{i}\in \partial
_{P}V(x_{i}),\,\,\xi _{i}\rightarrow \xi \,\,\,\text{as}\,\,i\rightarrow
\infty ;
\end{equation*}%
moreover, due to the upper semi-continuity of $\mathfrak{m}$ at $x$ and $%
\mathfrak{m}(x)<+\infty ,$ by assumption, we may\ assume up to a subsequence
that 
\begin{equation}
\mathfrak{m}(x_{i})\leq \mathfrak{m}(x)+\frac{1}{i}\,\,\forall i\in \mathbb{N%
}.  \label{lb.A}
\end{equation}%
By the Lipschitzianity of $F$ we also choose a sequence $(v_{i})_{i\geq 1}$
such that $v_{i}\in F(x_{i})$ and $v_{i}\rightarrow v.$ Since $(i)$ holds
with $\partial =\partial _{P}$, for each $i$ there exists $x_{i}^{\ast }\in
A(x_{i})\cap \mathrm{B}_{\left\Vert F(x_{i})\right\Vert +\mathfrak{m}%
(x_{i})} $ such that 
\begin{equation}
\langle \xi _{i},v_{i}-x_{i}^{\ast }\rangle +aV(x_{i})+W(x_{i})\leq 0.
\label{c.5.1a}
\end{equation}%
Then, since the\ maximal monotone operator $A$ has a closed graph, and $%
(x_{i}^{\ast })_{i}$ is bounded, we assume w.l.o.g. that 
\begin{equation*}
x_{i}^{\ast }\rightarrow x^{\ast }\in A(x)\cap \mathrm{B}_{\mathfrak{m}%
(x)}\,\,\text{as}\,\,i\rightarrow \infty .
\end{equation*}%
So, by passing to the limit in\ (\ref{c.5.1a}) as $i\rightarrow \infty $,
and using the lower semicontinuity of $W,$ we obtain that 
\begin{equation}
\langle \xi ,v-x^{\ast }\rangle +aV(x)+W(x)\leq 0,  \label{c.5.2}
\end{equation}%
which shows that $(i)$ holds when $\partial =\partial _{L}.$

$(i)\Rightarrow (V,W)$ is a strong $a$-Lyapunov pair for $(\ref{ieq}).$
According to Theorem \ref{StrongLyap} we only need to show that (\ref{s2})
holds. We fix $x\in \text{dom}V,$ $\xi \in \partial _{P,\infty }V(x)$ and $%
v\in F(x).$ There exist sequences $(x_{i})_{i},$ $(\xi _{i})_{i},$ and $%
(\alpha _{i})_{i}$ such that 
\begin{equation*}
x_{i}\overset{V}{\rightarrow }x,\,\,\xi _{i}\in \partial
_{P}V(x_{i}),\,\,\alpha _{i}\downarrow 0,\,\,\alpha _{i}\xi _{i}\rightarrow
\xi \ \text{as}\,\,i\rightarrow \infty .
\end{equation*}%
By arguing as in the last paragraph above there also exists a sequence $%
(v_{i})_{i}$\ such that $v_{i}\in F(x_{i})$ and $v_{i}\rightarrow v\ $as$\
i\rightarrow \infty .$ Moreover, using the current assumption on $A^{\circ
}, $ there exists $m>0$ such that $\sup_{i}\mathfrak{m}(x_{i})\leq m.$ Now,
by assumption $(ii),$ for each $i\in \mathbb{N}$ there exists a sequences $%
x_{i}^{\ast }\in A(x_{i})\cap \mathrm{B}_{\left\Vert F(x_{i})\right\Vert +%
\mathfrak{m}(x_{i})}\subset A(x_{i})\cap \mathrm{B}_{\left\Vert
F(x_{i})\right\Vert +m}$ and 
\begin{equation}
\langle \xi _{i},v_{i}-x_{i}^{\ast }\rangle +aV(x_{i})+W(x_{i})\leq 0.
\label{c.1.1}
\end{equation}%
By using\ again that $A$ has a closed graph, and that $x_{i}^{\ast
}\rightarrow x^{\ast }\in A(x),$ by multiplying the last inequality above\ (%
\ref{c.1.1}) by $\alpha _{i}$ and next taking limits as $i\rightarrow \infty 
$, we arrive at (\ref{s2}). The proof of the corollary is finished since $%
(ii)$ is a necessary condition for\ strong $a$-Lyapunov pairs, as we have
shown in the proof of Theorem \ref{StrongLyap}.
\end{proof}

\end{document}